\documentclass[12pt,a4paper]{amsart}
\usepackage{a4wide,amsmath,amsfonts,amssymb,amsthm,mathrsfs,setspace,bm,amsxtra,mathtools,mathrsfs,setspace,wasysym,textcomp} 
\mathtoolsset{showonlyrefs}
\usepackage[english]{babel}
\usepackage[all]{xy}
\usepackage{verbatim}
\usepackage{todonotes}

\newtheorem{thm}{Theorem}[section]
\newtheorem{prop}[thm]{Proposition}
\newtheorem{lemma}[thm]{Lemma}

\newtheorem{cor}[thm]{Corollary}

\numberwithin{equation}{section}
\newcommand{\coker}{\operatorname{coker}}
\newcommand{\Aut}{\operatorname{Aut}}
\newcommand{\rk}{\operatorname{rk}}

\newcommand{\Hom}{\operatorname{Hom}}
\newcommand{\End}{\operatorname{End}}

\newcommand{\GL}{\operatorname{GL}}
\newcommand{\im}{\operatorname{Im}}
\newcommand{\Pic}{\operatorname{Pic}}
\newcommand{\Ol}{\mathcal{O}}

\newcommand{\A}{\mathcal{A}}

\newcommand{\M}{\mathcal{M}}
\newcommand{\E}{\mathcal{E}}
\newcommand{\U}{\mathcal{U}}
\newcommand{\V}{\mathcal{V}}
\newcommand{\W}{\mathcal{W}}

\newcommand{\T}{\mathcal{T}}

\newcommand{\tot}{\operatorname{Tot}}
\newcommand{\Spec}{\operatorname{Spec}}

\newcommand{\Mk}{\M^n(r,a,c)}
\newcommand{\Pk}{P_{\vec{k}}}
\newcommand{\Lk}{L_{\vec{k}}}
\newcommand{\Gk}{G_{\vec{k}}}
\newcommand{\Z}{\mathbb{Z}}

\newcommand{\Com}{\mathbb{C}}
\newcommand{\Pu}{\mathbb{P}^1}
\newcommand{\On}{\Ol_{\Sigma_n}}

\newcommand{\Uk}{\U_{\vec{k}}}
\newcommand{\Vk}{\V_{\vec{k}}}
\newcommand{\Wk}{\W_{\vec{k}}}

\newcommand{\li}{\ell_\infty}
\newcommand{\Ui}{\U_{\vec{k},\infty}}
\newcommand{\Vi}{\V_{\vec{k},\infty}}
\newcommand{\Wi}{\W_{\vec{k},\infty}}

\newcommand{\Pkm}{P_{\vec{k},m}}
\newcommand{\uside}[1]{{}^{#1}\mspace{-3mu}}

\newcommand{\lhra}{\lhook\joinrel\relbar\joinrel\rightarrow}
\newcommand{\lra}{\longrightarrow}

\newcommand{\Stab}{\operatorname{Stab}}
\newcommand{\Hc}{\T(c)}

\newcommand{\Hilb}{\operatorname{Hilb}}

\renewcommand\leftmark{{\scshape ADHM data for the Hilbert scheme of  the total space of $\Ol_{\Pu}(-n)$}}
\renewcommand\rightmark{{\scshape ADHM data for the Hilbert scheme of  the total space of $\Ol_{\Pu}(-n)$}}

\title[ADHM data for the Hilbert scheme of  the total space of $\Ol_{\Pu}(-n)$]{ADHM data for the Hilbert scheme of points\\[5pt] of the total space of $\Ol_{\Pu}(-n)$}
\date{\today}

\begin{document}

\maketitle
\begin{center}{\sc Claudio Bartocci,$^\P$ Ugo Bruzzo,$^{\S\ddag}$ \\ Valeriano Lanza$^\P$ and Claudio L. S. Rava$^\P$ } \\[10pt]  \small 
$^\P$Dipartimento di Matematica, Universit\`a di Genova, \\Via Dodecaneso 35, 16146 Genova, Italia\\[3pt]
  $^{\S}$Scuola Internazionale Superiore di Studi Avanzati,\\   Via Bonomea 265, 34136 Trieste, Italia\\[3pt]
  $^\ddag$Istituto Nazionale di Fisica Nucleare, Sezione di Trieste
\end{center}
\bigskip
\begin{quote}\small {\sc Abstract.} Relying on a monadic description of the moduli space of framed sheaves on Hirzebruch surfaces, we construct ADHM data for the Hilbert scheme of points of the total space of the line bundle $\mathcal O(-n)$ on $\mathbb P^1$.
 \end{quote}
{\footnotesize
\setcounter{tocdepth}{1}
\tableofcontents
}
\section{Introduction}

Let $X$ be a smooth quasi-projective irreducible surface over $\Com$. The Hilbert scheme of points $\operatorname{Hilb}^c(X)$, which parameterizes $0$-dimensional subschemes of $X$ of length $c$, is well known to be quasi-projective \cite{groth61} and smooth of dimension $2c$ \cite{fog68}; indeed, the so-called Hilbert-Chow morphism $\operatorname{Hilb}^c(X)\longrightarrow S^c X$ onto the $c$-th symmetric product of $X$ is a resolution of singularities. 
Hilbert schemes of points on surfaces  were extensively studied from many perspectives over the past two decades (see e.g.~\cite{Nakabook,lehn2004,naka2014}), however there are relatively few cases in which they are susceptible of an explicit description. Arguably, the most significant examples are the spaces  $\operatorname{Hilb^c}(\Com^2)$, which can be described by means of  linear data, the so-called ADHM (Atiyah-Drinfel'd-Hitchin-Manin) data
\cite{Nakabook}. Also the Hilbert schemes of points of multi-blowups of $\mathbb C^2$ admit an ADHM description, as provided by the work of A.A.\ Henni \cite{henni} specialized to the rank one case.

The goal of this paper is to provide an ADHM-type construction for the Hilbert schemes of points over the total space $\tot(\Ol_{\Pu}(-n))$ of the line bundle $\Ol_{\Pu}(-n)$ on $\Pu$. These spaces are the rank 1 case of the moduli spaces of framed sheaves of the Hirzebruch surface $\Sigma_n$ (by framing to the trivial bundle on a divisor linearly equivalent to the section of $\Sigma_n\to \mathbb P^1$ of positive self-intersection) which were studied in \cite{BPT,bbr}. These modules spaces were considered in physics in connection with the so-called D4-D2-D0 brane system in topological string theory (cf.\ \cite{OSV,AOSV} and \cite{BPT} for a concise discussion).

To construct the ADHM data for the Hilbert scheme of points of $\operatorname{Hilb}^c(\tot(\Ol_{\Pu}(-n)))$ we identify it with the  moduli space $\M^{n}(1,0,c)$  of framed sheaves on the Hirzebruch surface $\Sigma_n$ that have rank $1$, vanishing first Chern class, and second Chern class $c_2 = c$, and exploit the description of $\M^{n}(1,0,c)$ in terms of monads given in \cite{bbr}. Theorem \ref{thm0} states that the moduli space $\M^{n}(1,0,c)$ is isomorphic to the quotient $P^n(c) / \GL(c, \Com)^{\times 2}$, where $P^n(c)$ is a quasi-affine variety contained in the linear space $\End(\Com^{c})^{\oplus n+2}\oplus\Hom(\Com^{c},\Com)$. This result relies on the fact  that the partial quotient 
$P^n(c) / \GL(c, \Com)$ can be assembled glueing  $c+1$ open sets, each one isomorphic to the space of ADHM data for  $\operatorname{Hilb^c}(\Com^2)$ (Theorem \ref{thm1}). Since the proof of Theorem \ref{thm1} is based on the description of the moduli spaces of framed sheaves on $\Sigma_n$ worked out in \cite{bbr}, for the reader's convenience we briefly recall here the fundamental ingredients of that construction.

 \subsection*{Acknowledgments} This work was partially supported by PRIN ``Geometria delle variet\`a algebriche", by the University of Genoa's project ``Aspetti matematici della teoria dei campi interagenti e quantizzazione per deformazione" and by GNSAGA-INDAM. U.B. is a member of the VBAC group.

\subsection*{Background material} Let $\Sigma_n$ be the $n$-th Hirzebruch surface, i.e., the projective closure of the total space of the line bundle $\mathcal{O}_{\mathbb{P}^1}(-n)$; we restrict ourselves to the case $n > 0$. We denote by $F$ the class in $\Pic(\Sigma_n)$  of the fibre of the natural ruling $\Sigma_n\longrightarrow\Pu$, by $H$ the class of the section of the ruling squaring to $n$, and by $E$    the class of the section squaring to $-n$. We fix a curve $\ell_{\infty}\simeq\Pu$ in $\Sigma_n$ linearly equivalent to $H$ and think of it  as the ``line at infinity''. 

A framed sheaf on $\Sigma_n$ is a pair $(\E, \theta)$, where $\E$ is a torsion-free sheaf that is trivial along  $\ell_{\infty}$, and $\theta \colon \E\vert_{\ell_{\infty}}\stackrel{\sim}{\longrightarrow}\Ol_{\li}^{\oplus r}$ is a fixed isomorphism, $r$ being  the rank of $\E$. A morphism between the framed sheaves $(\E, \theta)$, $(\E', \theta')$ is by definition a morphism $\Lambda\colon \E \longrightarrow \E'$ such that  
$\theta'\circ\Lambda\vert_{\ell_{\infty}} = \theta$. The moduli space parameterizing isomorphism classes of framed sheaves $(\E, \theta)$ on $\Sigma_n$ with $\textrm{ch}(\E) = (r, aE, -c -\frac{1}{2} na^2)$ will be denoted by $\M^{n}(r,a,c)$. We assume that the framed sheaves are normalized in such a way that $0\leq a\leq r-1$.

A description of the moduli space $\M^{n}(r,a,c)$ in terms of monads was provided in \cite{bbr},
generalizing work by Buchdahl \cite{Bu}.  If $[(\E, \theta)]$ lies in $\M^{n}(r,a,c)$,
the sheaf $\E$ is isomorphic to the cohomology of a monad 
\begin{equation}
\xymatrix{
M(\alpha,\beta):&0 \ar[r] & \Uk \ar[r]^-{\alpha} & \Vk \ar[r]^-{\beta} & \Wk \ar[r] & 0
}\,, \label{fundamentalmonad}
\end{equation}
where $\vec{k} = (n,r,a,c)$; in others words, the terms of \eqref{fundamentalmonad} depend only on the Chern character of $\E$. More precisely, if we put
\begin{equation}\begin{cases}
\begin{aligned}
k_1&=c+\dfrac{1}{2}na(a-1)\\
k_2&=k_1+na\\
k_3&=k_1+(n-1)a\\
k_4&=k_1+r-a\,,
\end{aligned}
\end{cases}\label{k_i}
\end{equation}
we have
\begin{equation}
\left\{
\begin{aligned}
\Uk&:=\On(0,-1)^{\oplus k_1}\\
\Vk&:=\On(1,-1)^{\oplus k_2} \oplus \On^{\oplus k_4}\\
\Wk&:=\On(1,0)^{\oplus k_3}\,.
\end{aligned}
\right.
\end{equation}
This procedure yields a map 
\begin{equation}
(\E,\theta)\longmapsto\Hom(\Uk,\Vk)\oplus\Hom(\Vk,\Wk)\,,
\label{eqbbr1}
\end{equation}
whose image $\Lk$ is a smooth variety, which can be completely characterized by imposing suitable conditions on the pairs $(\alpha, \beta) \in \Hom(\Uk,\Vk)\oplus\Hom(\Vk,\Wk)$ \cite[\S2]{bbr}. 
One can construct a principal $\GL(r,\Com)$-bundle $\Pk$ over $\Lk$ whose fibre over a point $(\alpha,\beta)$ is naturally identified with the space of framings for the cohomology of the complex \eqref{fundamentalmonad}. Hence, the map \eqref{eqbbr1} can be lifted to a map
\begin{equation*}
(\E,\theta)\longmapsto\theta\in\Pk\,.
%\label{eqbbr1}
\end{equation*}
The algebraic group $\Gk=\Aut(\Uk)\times\Aut(\Vk)\times\Aut(\Wk)$ of isomorphisms of monads of the form 
\eqref{fundamentalmonad} acts freely on $\Pk$, and the moduli space $\M^{n}(r,a,c)$ can be described as the quotient $\Pk/\Gk$ \cite[Theorem 3.4]{bbr}. This space is nonempty if and only if $c + \frac{1}{2} na(a-1) \geq 0$, and, in this case, is a smooth algebraic variety of dimension $rc + (r-1) na^2$.

If the sheaf $\E$ has rank $r=1$, by normalizing we can assume $a=0$. Hence, the double dual $\E^{\ast\ast}$ of $\E$, being locally free with $\textrm{c}_1(\E^{\ast\ast}) = \textrm{c}_1(\E)= 0$, is isomorphic to structure   sheaf $\On$. As a consequence, since $\E$ is trivial on $\ell_{\infty}$, the correspondence 
$$\E  \longmapsto \mbox{schematic support of}\ \E^{\ast\ast}/ \E$$  yields an isomorphism
\begin{equation} 
\M^{n}(1,0,c) \simeq \operatorname{Hilb}^c (\Sigma_n \setminus \ell_{\infty}) = \operatorname{Hilb}^c (\tot(\Ol_{\Pu}(-n)))\,.
\end{equation}
In the following, we shall denote the moduli space  $\M^{n}(1,0,c)$ simply by $\M^{n}(c)$.

\section{Statement of the Main Theorem}\label{sectionmain}
We call $P^{n}(c)$ the subset of the vector space $\End(\Com^{c})^{\oplus n+2}\oplus\Hom(\Com^{c},\Com)$ whose points $\left(A_1,A_2;C_1,\dots,C_{n};e\right)$ satisfy the following conditions:
\begin{enumerate}
 \item[(P1)]
\begin{equation*}
\begin{cases}
A_1C_1A_{2}=A_2C_{1}A_{1}&\qquad\text{when $n=1$}\\[15pt]
\begin{aligned}
A_1C_q&=A_2C_{q+1}\\
C_qA_1&=C_{q+1}A_2
\end{aligned}
\qquad\text{for}\quad q=1,\dots,n-1&\qquad\text{when $n>1$}
\end{cases}\,;
\end{equation*} \smallskip
\item[(P2)]
there exists $[\nu_{1},\nu_{2}]\in\Pu$ such that $\det(\nu_1A_1+\nu_2A_2)\neq0$;\smallskip
\item[(P3)]
for all values of the parameters $\left([\lambda_1,\lambda_2],(\mu_1,\mu_{2})\right)\in\Pu\times\Com^{2}$ such that
\begin{equation*}
\lambda_{1}^{n}\mu_{1}+\lambda_{2}^{n}\mu_{2}=0
\end{equation*}
there is no nonzero vector $v\in\Com^c$   such that 
\begin{equation*}
\left\{
\begin{aligned}
\left(\lambda_2{A_1}+\lambda_1{A_2}\right)v&=0\\
(C_{1}A_{2}+\mu_{1}\bm{1}_{c})v&=0\\
(C_{n}A_{1}+(-1)^{n-1}\mu_{2}\bm{1}_{c})v&=0\\
ev&=0\,.
\end{aligned}\right.
\end{equation*}
\end{enumerate}
We define an action of 
$\GL(c, \Com)^{\times 2}$ on $P^n(c)$ by the equations
\begin{equation}
\left\{
\begin{array}{ccll}
C_j & \longmapsto & \phi_1C_j\phi_2^{-1} &\qquad j=1,\dots,n\\
A_i & \longmapsto & \phi_2A_i\phi_1^{-1} &\qquad i=1,2\\
e & \longmapsto & e\phi_1^{-1}&
\end{array}
\right.\qquad (\phi_1,\phi_2)\in\GL(c, \Com)^{\times 2}\,.
\label{eqrho}
\end{equation}
\begin{thm}
\label{thm0}
There is an isomorphism of complex varieties 
\begin{equation*}
P^n(c)\left/\GL(c, \Com)^{\times 2}\right.\simeq \M^n(c)=\operatorname{Hilb}^c(\tot(\Ol_{\Pu}(-n)))\,,
\end{equation*}
and $P^{n}(c)$ is a locally trivial principal $\GL(c, \Com)^{\times 2}$-bundle over $\M^n(c)$.
\end{thm}

\subsection{A consistency check}
Before proving Theorem \ref{thm0} we check its consistency in the simplest case  $c=1$, by verifying that the quotient $P^{n}(1)/(\Com^{*})^{\times2}$ is isomorphic to the total space of $\Ol_{\Pu}(-n)$. Indeed, one has $\tot(\Ol_{\Pu}(-n))\simeq\widetilde{T}_{n}/\Com^{*}$, where
\begin{equation*}
\widetilde{T}_{n}=\left\{\left.((y_{1},y_{2}),(u_{1},u_{2}))\in\left(\Com^{2}\setminus \{0\} \right)\times\Com^{2}\right|u_{1}y_{1}^{n}=u_{2}y_{2}^{n}\right\}
\end{equation*}
and the $\Com^{*}$-action is
\begin{equation*}
\left\{
\begin{aligned}
(y_{1},y_{2})&\longmapsto\lambda(y_{1},y_{2})\\
(u_{1},u_{2})&\longmapsto(u_{1},u_{2})
\end{aligned}
\right.\qquad\lambda\in\Com^{*}
\end{equation*}
(cf.~eq.~\eqref{eqSn}).
\begin{prop}
\label{propc=1}
$
P^{n}(1)/(\Com^{*})^{\times2}\simeq \tot(\Ol_{\Pu}(-n))\,.
$
\end{prop}
\begin{proof}
When $c=1$, the matrices $(A_{1},A_{2},C_{1},\dots,C_{n},e)$ are complex numbers, and condition (P2) is equivalent to requiring that $(A_{1},A_{2})\neq(0,0)$. When $n=1$ condition (P1) is identically satisfied, while when $n>1$ it is equivalent to 
\begin{equation*}
\begin{cases}
C_q=\left(\frac{A_{2}}{A_{1}}\right)^{n-q}C_{n}
\qquad\text{for}\quad q=1,\dots,n-1&  \qquad \text{if}\quad A_{1}\neq0\\[7pt]
C_q=\left(\frac{A_{1}}{A_{2}}\right)^{q-1}C_{1}
\qquad\text{for}\quad q=2,\dots,n&\qquad \text{if}\quad A_{2}\neq 0\,.
\end{cases}
\end{equation*}
Using these equations it is possible to show that condition (P3) reduces to    $e\neq0$.
By acting with $\left(\Com^{*}\right)^{\times 2}$ we can fix $e=1$, and the maximal subgroup preserving this condition is clearly $\{1\}\times\Com^{*}$.
We introduce the variety
\begin{equation*}
\widetilde{Y}_{n}=\left\{\left.((y_{1},y_{2}),(x_{1},x_{2}))\in
\left(\Com^{2}\setminus\{0\}\right)\times\Com^{2}\right|x_{1}y_{1}^{n-1}=x_{2}y_{2}^{n-1}\right\}\,,
\end{equation*}
with $n\geq1$, and we let $\Com^{*}$ act on $\widetilde{Y}_{n}$ as follows:
\begin{equation*}
\left\{
\begin{aligned}
(y_{1},y_{2})&\longmapsto\lambda(y_{1},y_{2})\\
(x_{1},x_{2})&\longmapsto\lambda^{-1}(x_{1},x_{2})
\end{aligned}\qquad \lambda\in\Com^{*}
\right.\,.
\end{equation*}
We cover $\widetilde{Y}_{n}$ with the two $\Com^{*}$-invariant subsets $\widetilde{Y}_{n,i}=\{y_{i}\neq0\}$, and analogously we cover $P^{n}(1)$ with the $(\Com^{*})^{\times 2}$-invariant subsets $P^{n}(1)_{i}=\{A_{i}\neq0\}$, $i=1,2$.
Next, we define the morphisms
\begin{equation*}
\begin{array}{ccl}
\widetilde{Y}_{n,i}&\longrightarrow&P^{n}(1)_{i}\\[5pt]
((y_{1},y_{2}),(x_{1},x_{2}))&\longmapsto&
\begin{cases}
\left(y_{1},y_{2},\left(\frac{y_{2}}{y_{1}}\right)^{n-1}x_{2},\left(\frac{y_{2}}{y_{1}}\right)^{n-2}x_{2},\ldots,x_{2},1\right) & i=1\\[8pt]
\left(y_{1},y_{2},x_{1},\left(\frac{y_{1}}{y_{2}}\right)x_{1},\ldots,\left(\frac{y_{1}}{y_{2}}\right)^{n-1}x_{1},1\right) & i=2\,.
\end{cases}
\end{array}
\end{equation*}
These   glue together providing a $\Com^{*}$-equivariant closed immersion $\widetilde{Y}_{n}\lhra P^{n}(1)$, which induces an isomorphism 
\begin{equation*}
P^{n}(1)/(\Com^{*})^{\times 2}\simeq\widetilde{Y}_{n}/\Com^{*}\,.
\end{equation*}
Finally, the  $\Com^{*}$-equivariant  morphism
\begin{equation*}
\begin{array}{ccl}
\widetilde{Y}_{n}&\longrightarrow&\left(\Com^{2}\setminus\{0\}\right)\times\Com^{2}\\
((y_{1},y_{2}),(x_{1},x_{2}))&\longmapsto&((y_{1},y_{2}),(u_{1},u_{2}))=((y_{1},y_{2}),(x_{1}y_{2},x_{2}y_{1}))\,.
\end{array}
\end{equation*}
establishes the required isomorphism.
\end{proof}

\section{Glueing ADHM data}
\label{secLoc}
In this section we provide an ADHM description for each open set of a suitable open cover of $\M^{n}(c)$. 
If we fix $c+1$ distinct fibres $f_0,\dots,f_{c}\in |F|$, for any $[(\E,\theta)]\in \M^{n}(c)$ there exists at least one $m\in\{0,\dots,c\}$ such that $\E|_{f_{m}}\simeq\Ol_{f_{m}}$. With this in mind, we choose the fibres $f_{m}$ cut in
\begin{equation}
\Sigma_n = \left\{([y_1,y_2],[x_1,x_2,x_3])\in\mathbb{P}^1\times\mathbb{P}^2\;|\;x_1y_1^n=x_2y_2^n \right\}
\label{eqSn}
\end{equation}
by the equations
\begin{equation}
f_m=\{[y_1,y_2]=[c_m,s_m]\}\qquad m=0,\dots,c
\label{eq12}
\end{equation}
where 
\begin{equation}
c_m=\cos\left(\pi\frac{m}{c+1}\right)\qquad\text{,}\qquad s_m=\sin\left(\pi\frac{m}{c+1}\right)\,.
\label{eqcmsm}
\end{equation}
Then we get an open cover $\left\{\M^n(c)_m\right\}_{m=0}^{c}$ for $\M^n(c)$ by letting
\begin{equation*}
\M^n(c)_m:=\left\{[(\E,\theta)]\in\M^n(c)\left|
\begin{array}{l}
\text{the restricted sheaf $\E|_{f_m}$}\\
\text{is isomorphic to $\Ol_{f_m}$}
\end{array}
\right.\right\}.
\end{equation*}
Each of these spaces is isomorphic to  the Hilbert scheme of points of $\Com^2$, so that it admits the ADHM description \cite{Nakabook},
which we briefly recall.  The variety $\Hc$ of ADHM data is defined as the space of triples $(b_1,b_2,e)\in\End(\Com^c)^{\oplus 2}\oplus\Hom(\Com^{c},\Com)$ such that
\begin{itemize}
\item[(T1)]
$[b_{1},b_{2}]=0\,;
$ \smallskip
\item[(T2)]
for all $(z,w)\in\Com^2$ there is no nonzero vector $v\in\Com^{c}$ such that
\begin{equation*}
\left\{
\begin{aligned}
(b_{1}+z\bm{1}_c)v&=0\\
(b_{2}+w\bm{1}_c)v&=0\\
ev&=0\,.
\end{aligned}
\right.
\end{equation*}
\end{itemize}\smallskip

A $\GL(c, \Com)$-action on $\Hc$ is naturally defined as follows:
\begin{equation}
\left\{
\begin{array}{rcl}
b_{i}&\longmapsto & \phi b_{i}\phi^{-1}\qquad i=1,2 \\
e&\longmapsto & e\phi^{-1}\\
\end{array}
\right.
\qquad\phi\in\GL(c, \Com)\,.
\label{eqGL(c)onK}
\end{equation} 
\smallskip

The ADHM data for the open set $\M^{n}(c)_{m}$ will be denoted by $(b_{1m},b_{2m},e_{m})$; 
the transition functions on the intersections $\M^n(c)_{ml}=\M^n(c)_m\cap\M^n(c)_l$
are explicitly described in the next Theorem.
\begin{thm}
\label{thm1}
The intersection $\M^n(c)_{ml}=\M^n(c)_m\cap\M^n(c)_l$ is characterized by the condition
\begin{equation*}
\det\left(c_{m-l}\bm{1}_c+s_{m-l}b_{1l}\right) \neq0 \quad(\text{or, equivalently,}\ 
\det\left(c_{l-m}\bm{1}_c+s_{l-m}b_{1m}\right) \neq0)\,,
\end{equation*}
where $c_{m}$ and $s_{m}$ are the numbers defined in eq.~\eqref{eqcmsm}. On any of these intersections, the ADHM data are related by the transition functions
\begin{equation*}
\begin{aligned}
\varphi_{lm}\colon \M^n(c)_{ml} &\longrightarrow \M^n(c)_{ml}\\
\left[(
b_{1m}, b_{2m}, e_{m})\right]&\longmapsto 
\left[(b_{1l},b_{2l},e_{l})\right]\,,
\end{aligned}
\end{equation*}
\begin{equation*}
\text{where}\qquad
\left\{
\begin{aligned}
b_{1l}&=\left(c_{m-l}\bm{1}_c-s_{m-l}b_{1m}\right)^{-1}\left(s_{m-l}\bm{1}_c+c_{m-l}b_{1m}\right)\\
b_{2l}&=\left(c_{m-l}\bm{1}_c-s_{m-l}b_{1m}\right)^{n}b_{2m}\\
e_{l}&=e_{m}\,.
\end{aligned}
\right.
\end{equation*}
\end{thm}
To prove Theorem \ref{thm1} we observe that $\GL(c, \Com)$ can be embedded  as a closed subgroup of $\Gk$  by means of the homomorphism
\begin{equation}
\begin{array}{rccl}
\iota\colon &\GL(c, \Com)&\longrightarrow&\Gk\\[8pt]
&\phi&\longmapsto&\left(\uside{t}\phi^{-1},
\begin{pmatrix}
\uside{t}\phi^{-1} & 0 & 0\\
0 & \uside{t}\phi^{-1} & 0\\
0 & 0 & 1
\end{pmatrix},
\uside{t}\phi^{-1}\right)\,.
\end{array}
\label{eqiota}
\end{equation}

Let $\pi\colon \Pk \longrightarrow \M^n(c)$ be the canonical projection. The open subsets 
\begin{equation*}
P_{\vec{k},m}=\pi^{-1}\left(\M^n(c)_m\right)\,,\qquad m=0,\dots,c\,,
\end{equation*}
provide a $\Gk$-invariant open cover of $\Pk$; 
$\GL(c, \Com)$ acts on each $P_{\vec{k},m}$ via the immersion \eqref{eqiota}.
\begin{prop} 
There are $\GL(c, \Com)$-equivariant closed immersions
\begin{equation*}
j_{m}\colon \Hc\lhra P_{\vec{k},m}\,\qquad\text{for $m=0,\dots,c$}\,.
\end{equation*}
These  induce isomorphisms 
\begin{equation}\eta_m\colon\Hc/\GL(c, \Com)\longrightarrow P_{\vec{k},m}/\Gk\simeq\M^n(c)_m\qquad\text{for $m=0,\dots,c$}\,.
\label{eqeta}
\end{equation}
\label{proopen}
\end{prop}
\begin{proof}
See Section \ref{proof1}.
\end{proof}
We introduce the open subsets
\begin{equation*}
\Hc_{m,l}=j_{m}^{-1}\left(\im j_{m}\cap P_{\vec{k},l}\right)\qquad\text{for}\quad m,l=0,\dots,c\,.
\end{equation*}
\begin{lemma}
$
\Hc_{m,l}=\left\{(b_{1},b_{2},e)\in \Hc\left| \det\left(c_{m-l}\bm{1}_c-s_{m-l}b_{1}\right)\neq0\right.\right\}\,.
$
\end{lemma}
\begin{proof}
The intersection $\im j_{m}\cap P_{\vec{k},l}$ is   the set of points $(\alpha,\beta,\xi)\in \im j_m$ such that $\det\left(\left.\beta_1\right|_{f_l}\right)$ $\neq0$, where $\beta_1$ is the first component  of $\beta$. From the fact that $(\alpha,\beta,\xi)\in \im j_m$ it follows that
\begin{equation*}
\beta_1=\bm{1}_cy_{1m}+\uside{t}\mspace{2mu}b_{1}y_{2m}=
\begin{pmatrix}
\bm{1}_c & \uside{t}\mspace{2mu}b_{1}
\end{pmatrix}
\begin{pmatrix}
y_{1m}\\
y_{2m}
\end{pmatrix}=
\begin{pmatrix}
\bm{1}_c & \uside{t}\mspace{2mu}b_{1}
\end{pmatrix}
\begin{pmatrix}
c_m & s_m\\
-s_m & c_m
\end{pmatrix}
\begin{pmatrix}
y_{1}\\
y_{2}
\end{pmatrix}\,.
\end{equation*}
Since   $[y_{1},y_{2}]=[c_l,s_l]$ along $f_{l}$, the thesis follows.
\end{proof}
\begin{prop}
\label{pro1}
The map
\begin{equation}
\begin{array}{rccl}
\tilde{\varphi}_{lm}\colon &\Hc_{m,l}&\longrightarrow&\Hc_{l,m}\\
&
\begin{pmatrix}
b_{1}\\
b_{2}\\
e
\end{pmatrix}&\longmapsto&
\begin{pmatrix}
\left(c_{m-l}\bm{1}_c-s_{m-l}b_{1}\right)^{-1}\left(s_{m-l}\bm{1}_c+c_{m-l}b_{1}\right)\\
\left(c_{m-l}\bm{1}_c-s_{m-l}b_{1}\right)^{n}b_{2}\\
e
\end{pmatrix}
\end{array}
\label{eqfiml}
\end{equation}
is   $\GL(c, \Com)$-equivariant,  and induces an isomorphism 
$$\varphi_{lm}\colon \Hc_{m,l}/\GL(c, \Com)\lra \Hc_{l,m}/\GL(c, \Com)\,,$$ 
such that the triangle
\begin{equation}
\xymatrix{
\Hc_{m,l}/\GL(c, \Com) \ar[r]^-{\varphi_{lm}} \ar[rd]_-{\eta_{m,l}} & \Hc_{l,m}/\GL(c, \Com) \ar[d]^{\eta_{l,m}}\\
& \M^n(c)_{ml}
}
\label{eqtri}
\end{equation}
is commutative, where $\eta_{m,l}$  is the restriction of $\eta_{m}$ to $\Hc_{m,l}/\GL(c, \Com)$ (see eq.~\eqref{eqeta}).
\end{prop}
\begin{proof}
See Section \ref{proof2}.
\end{proof}
Theorem \ref{thm1} is now a direct consequence of Proposition \ref{pro1}.

\section{Proof of the Main Theorem}
We introduce the matrices
\begin{equation}
\begin{aligned}
A_{1m}&=c_mA_1-s_mA_2\,,\\
A_{2m}&=s_mA_1+c_mA_2\,,\\
E_{m}&=\left[\sum_{q=1}^{n}\binom{n-1}{q-1}c_{m}^{n-q}s_{m}^{q-1}C_q\right]A_{2m}\,,
\end{aligned}
\label{eq410}
\end{equation}
for $m=0,\dots,c$. Since the polynomial $\det(A_1\nu_{1}+A_2\nu_{2})$ has at most $c$ distinct roots in $\Pu$, the $\GL(c, \Com)^{\times 2}$-invariant open subsets
\begin{equation}
P^n(c)_m=\left\{\left(A_1,A_2,C_{1},\dots,C_{n},e\right)\in P^n(c)\left|
\det A_{2m}\neq0
\right.\right\}\,,\qquad m=0,\dots,c\,,
\label{pcnm}
\end{equation}
cover $P^{n}(c)$.
On $P^{n}(c)_{m}$ one can also define  the matrix
\begin{equation}
B_{m}=A_{2m}^{-1}A_{1m}\,.
\label{eqBm}
\end{equation}
The matrices $(B_{m},E_{m},A_{2m},e)$ provide local affine coordinates for $P^n(c)$.
\begin{prop}
\label{lmTomega}
The morphism
\begin{equation*}
\begin{array}{rccl}
\zeta_{m}\colon&P^{n}(c)_{m}&\longrightarrow&\left[\End(\Com^{c})^{\oplus 2}\oplus\Hom(\Com^{c},\Com)\right]\times\GL(c, \Com)\\
&(A_{1},A_{2};C_{1},\dots,C_{n};e)&\longmapsto&
\left(B_{m},E_{m},e;A_{2m}\right)
\end{array}
\end{equation*}
is an isomorphism onto $\Hc\times\GL(c, \Com)$. The induced $\GL(c, \Com)^{\times 2}$-action is given by
\begin{equation}
\left\{
\begin{array}{rcl}
B_{m} & \longmapsto & \phi_{1}B_{m}\phi_{1}^{-1}\\
E_{m} & \longmapsto & \phi_{1}E_{m}\phi_{1}^{-1}\\
A_{2m} & \longmapsto & \phi_{2}A_{2m}\phi_{1}^{-1}\\
e & \longmapsto & e\phi_{1}^{-1}\,.
\end{array}
\right.
\label{eqrhol}
\end{equation}
\end{prop}
We divide the proof of this Proposition into several steps. First we define the matrices $\sigma^h_{m}=(\sigma^{h}_{m;pq})_{0\leq p,q \leq h}$ for all $h\geq0$ and $m\in\Z$ by means of the equations
\begin{equation}
(s_{m}\mu_{1}+c_{m}\mu_{2})^{p}(c_{m}\mu_{1}-s_{m}\mu_{2})^{h-p}=\sum_{q=0}^{h}\sigma^{h}_{m;pq}\mu_{2}^{q}\mu_{1}^{h-q}
\label{eqsigma}
\end{equation}
for any $(\mu_{1},\mu_{2})\in\Com^{2}$ and $p=0,\dots,h$. Notice that $\sigma^h_{m}\sigma^h_{l}=\sigma^h_{m+l}$ and $\sigma^h_{0}=\bm{1}_{h+1}$. In particular, $\sigma^h_{m}$  is invertible for all $h\geq0$ and $m\in\Z$.

To prove the injectivity of $\zeta_{m}$ --- which is trivial only when $n=1$ --- we need the following Lemma.
\begin{lemma}
\label{propSyst}
Assume   $n>1$. If the matrices $A_{1},A_{2}\in\End(\Com^{c})$ satisfy the condition \mbox{\rm(P2)},   the system
\begin{equation}
\begin{pmatrix}
A_{1} & -A_{2}\\
& \ddots & \ddots\\
& & A_{1} & -A_{2}
\end{pmatrix}
\begin{pmatrix}
C_{1}\\
\vdots\\
\vdots\\
C_{n}
\end{pmatrix}=0\,,
\label{eq67}
\end{equation}
with $C_{q}\in\End(\Com^{c})$,  has      maximal rank, namely, $(n-1)c^{2}$. In particular, if $\det A_{2m}\neq0$, the general solution is
\begin{equation}
\begin{pmatrix}
C_{1}\\
\vdots\\
\vdots\\
C_{n}
\end{pmatrix}=(\sigma^{n-1}_{m}\otimes\bm{1}_{c})
\begin{pmatrix}
\bm{1}_{c}\\
B_{m}\\
\vdots\\
B_{m}^{n-1}
\end{pmatrix}
D_m\,,
\label{eq64}
\end{equation}
where we have chosen as free parameter the  matrix
\begin{equation}
D_m=
\sum_{q=1}^{n}\binom{n-1}{q-1}c_{m}^{n-q}s_{m}^{q-1}C_q\,.
\label{eq43}
\end{equation}
\end{lemma}
\begin{proof}
First we show by induction that the $(n-1)c\times nc$ matrices
\begin{equation*}
\mathscr{A}_{n}=\begin{pmatrix}
A_{1} & -A_{2}\\
& \ddots & \ddots\\
& & A_{1} & -A_{2}
\end{pmatrix}\qquad
\mathscr{A}'_{n}=\begin{pmatrix}
-{}^{t}\mspace{-2mu}A_{2} & {}^{t}\mspace{-2mu}A_{1}\\
& \ddots & \ddots\\
& & -{}^{t}\mspace{-2mu}A_{2} & {}^{t}\mspace{-2mu}A_{1}
\end{pmatrix}
\end{equation*}
have maximal rank for all $n>1$.  For $n=2$ condition (P2) ensures the existence of a point $[\nu_{1},\nu_{2}]\in\Pu$ such that $\det(A_{1}\nu_{1}+A_{2}\nu_{2})
\neq0$; it follows that the columns of $A_{1}$ and $A_{2}$ span a vector space of dimension $c$, so that $\rk\mathscr{A}_{2}=c$. The case of $\mathscr{A}'_{2}$ is analogous.

 {Assume that} the claim holds true for some $k>1$,  {and observe that}
\begin{equation}
\mathscr{A}_{k+1}=
\left(
\begin{array}{c|ccc|c}
\begin{matrix}
A_{1}\\
0\\
\vdots\\
0
\end{matrix} & &
{}^{t}\mathscr{A}'_{k} & &
\begin{matrix}
0\\
\vdots\\
0\\
-A_{2}\\
\end{matrix}
\end{array}
\right)\,.
\label{eqAind}
\end{equation}
Let $v\in\Com^{(k+1)c}$, and decompose it as
\begin{equation*}
v=
\begin{pmatrix}
v_{1}\\
\begin{matrix}
\\[-10pt]
v_{2}\\[-10pt]
\\
\end{matrix}
\\
v_{3}
\end{pmatrix}
\begin{array}{l}
\updownarrow c\\
\left\updownarrow 
\begin{matrix}
\\[-10pt]
(k-1)c\\[-10pt]
\\
\end{matrix}
\right.\\
\updownarrow c\\
\end{array}\,.
\end{equation*}
If $\mathscr{A}_{k+1}v=0$, we get
\begin{equation*}
\left\{
\begin{aligned}
A_{1}v_{1}+{}^{t}\mathscr{A}'_{k}v_{2}&=0\\
{}^{t}\mathscr{A}'_{k}v_{2}&=0\\
{}^{t}\mathscr{A}'_{k}v_{2}-A_{2}v_{3}&=0\,.
\end{aligned}
\right.
\end{equation*}
Since $\ker{}^{t}\mathscr{A}'_{k}=0$ by inductive hypothesis,  it follows that $\mathscr{A}_{k+1}$ has maximal rank. The case of $\mathscr{A}'_{k+1}$ is analogous.
Eq.~\eqref{eq64} is  checked by direct computation and eq.~\eqref{eq43} is obtained by using the invertibility of $\sigma^{n-1}_{m}$.
\end{proof}
Since $E_{m}=D_{m}A_{2m}$, the morphism $\zeta_{m}$ is injective.

Next we prove that $\im\zeta_{m}\subseteq\Hc\times\GL(c, \Com)$ via the following two Lemmas.
\begin{lemma}
\label{lmP1}
For all $(B_{m},E_{m},e;A_{2m})\in\im\zeta_{m}$ one has
\begin{equation}
[B_{m},E_{m}]=0\,.
\label{eqComm}
\end{equation}
\end{lemma}
\begin{proof}
For all $n\geq1$ condition (P1) implies that
\begin{equation*}
A_1C_qA_2-A_2C_qA_1=0\qquad \text{for}\quad q=1,\dots,n\,.
\end{equation*}
By recalling eqs. \eqref{eq410} and \eqref{eqBm}, the thesis follows from the identity
\begin{equation*}
A_{1}CA_{2}-A_{2}CA_{1}=A_{1m}CA_{2m}-A_{2m}CA_{1m}\,,
\end{equation*}
which holds true for all $C\in\End(\Com^{c})$ and for $m=0,\dots,c$.
\end{proof}
\begin{lemma}
 Let $(A_{1},A_{2};C_{1},\dots,C_{n};e)\in\End(\Com^{c})^{\oplus(n+2)}\oplus\Hom(\Com^{c},\Com)$ be an $(n+3)$-tuple which satisfies the condition \mbox{\rm(P1)} and $\det A_{2m}\neq0$. Then
 \begin{itemize}
 \item 
if $[\lambda_1,\lambda_2]=[c_m,s_m]$, the condition \mbox{\rm(P3)} is trivially satisfied; 
\item
if $[\lambda_1,\lambda_2]\neq[c_m,s_m]$, the condition \mbox{\rm(P3)} holds true  if and only if the condition \mbox{\rm(T2)} holds true for the triple $(B_{m},E_{m},e)$.
\end{itemize}
\label{lmP3}
\end{lemma}
\begin{proof}
One has
\begin{equation}
\lambda_{2}A_{1}+\lambda_{1}A_{2}=
\begin{cases}
\lambda A_{2m} &\qquad \text{if}\quad [\lambda_1,\lambda_2]=[c_m,s_m]\\
\lambda A_{2m}(B_{m}+z\bm{1}_{c}) &\qquad \text{if}\quad [\lambda_1,\lambda_2]\neq[c_m,s_m]
\end{cases}
\label{eqlAlA}
\end{equation}
for some $\lambda\neq0$, where
\begin{equation*}
z=\frac{c_{m}\lambda_{1}+s_{m}\lambda_{2}}{-s_{m}\lambda_{1}+c_{m}\lambda_{2}}\,.
\end{equation*}
This proves the first statement.
As for the second statement, eq.~\eqref{eq64} yields
\begin{equation}
\begin{aligned}
C_{1}&=(c_{m}\bm{1}_{c}-s_{m}B_{m})^{n-1}E_{m}A_{2m}^{-1}\\
C_{n}&=(s_{m}\bm{1}_{c}+c_{m}B_{m})^{n-1}E_{m}A_{2m}^{-1}\,.
\end{aligned}
\label{eq406}
\end{equation}
Moreover, whenever $[\lambda_1,\lambda_2]\neq[c_m,s_m]$, the condition $\lambda_{1}^{n}\mu_{1}+\lambda_{2}^{n}\mu_{2}=0$ is satisfied if and only if
\begin{equation}
\left\{
\begin{aligned}
\mu_{1}&=(s_{m}z+c_{m})^{n}w\\
\mu_{2}&=-(c_{m}z+s_{m})^{n}w
\end{aligned}
\right.
\label{eqmu}
\end{equation}
for some $w\in\Com$.
Eqs.~\eqref{eqlAlA} and \eqref{eqmu} show the equivalence of the following systems:
\begin{equation*}
\left\{
\begin{aligned}
\left(\lambda_2{A_1}+\lambda_1{A_2}\right)v&=0\\
(C_{1}A_{2}+\mu_{1}\bm{1}_{c})v&=0\\
(C_{n}A_{1}+(-1)^{n-1}\mu_{2}\bm{1}_{c})v&=0
\end{aligned}\right.\qquad\Longleftrightarrow\qquad
\left\{
\begin{aligned}
(B_{m}+z\bm{1}_{c})v&=0\\
(s_{m}z+c_{m})^{n}(E_{m}+w\bm{1}_{c})v&=0\\
(-c_{m}z+s_{m})^{n}(E_{m}+w\bm{1}_{c})v&=0\,.
\end{aligned}
\right.
\end{equation*}
Since the polynomials $s_{m}z+c_{m}$ and $-c_{m}z+s_{m}$ are coprime in $\Com[z]$, the right-hand system is equivalent to
\begin{equation*}
\left\{
\begin{aligned}
(B_{m}+z\bm{1}_{c})v&=0\\
(E_{m}+w\bm{1}_{c})v&=0\,.
\end{aligned}
\right.
\end{equation*}
\end{proof}
Finally we prove that $\Hc\times\GL(c, \Com)\subseteq\im\zeta_{m}$. Let $(b_{1},b_{2},e;A)\in \Hc\times\GL(c, \Com)$; if we set
\begin{equation*}
\begin{aligned}
A_{1}&=A(c_{m}b_{1}+s_{m}\bm{1}_{c})\,,\\
A_{2}&=A(-s_{m}b_{1}+c_{m}\bm{1}_{c})\,,
\end{aligned}
\end{equation*}
\begin{equation}
\begin{pmatrix}
C_{1}\\
\vdots\\
\vdots\\
C_{n}
\end{pmatrix} =(\sigma^{n-1}_{m}\otimes\bm{1}_{c})
\begin{pmatrix}
\bm{1}_{c}\\
b_{1}\\
\vdots\\
b_{1}^{n-1}
\end{pmatrix}
b_{2}A^{-1}\,,
\label{eqACe}
\end{equation}
then $(A_1,A_2;C_1,\dots,C_n;e)\in P^n(c)_m$ and $\zeta_{m}(A_1,A_2;C_1,\dots,C_n;e)=(b_{1},b_{2},e;A)$. It is an easy matter to verify by substitution that the condition (P1) holds. Notice now that by substituting \eqref{eqACe} into eq.~\eqref{eq410} one gets
\begin{equation*}
A_{1m}=Ab_{1}\qquad\text{,}\qquad
A_{2m}=A\qquad\text{,}\qquad
E_{m}=b_{2}\,.
\end{equation*}
This shows that $A_{2m}$ is invertible, and in particular the condition (P2) holds true. By eq.~\eqref{eqBm} one has that $B_{m}=b_{1}$, so that the validity of the condition (P3) follows from the condition (T2)  by Lemma \ref{lmP3}. This concludes the proof of Proposition \ref{lmTomega}.\qed

We now compute the transition functions on the intersections $P^{n}(c)_{ml}=P^{n}(c)_{m}\cap P^{n}(c)_{l}$, for $m,l=0,\dots,c$. First observe that
\begin{equation*}
\zeta_{m}\left(P^{n}(c)_{ml}\right)=\Hc_{m,l}\times\GL(c, \Com)\,.
\end{equation*}
This fact is a consequence of the identity
\begin{equation}
\begin{split}
A_{2l}&=
\begin{pmatrix}
s_{l}\bm{1}_{c} & c_{l}\bm{1}_{c}
\end{pmatrix}
\begin{pmatrix}
c_{m}\bm{1}_{c} & s_{m}\bm{1}_{c}\\
-s_{m}\bm{1}_{c} & c_{m}\bm{1}_{c}
\end{pmatrix}
\begin{pmatrix}
A_{1m}\\
A_{2m}
\end{pmatrix}
=A_{2m}(c_{m-l}\bm{1}_{c}-s_{m-l}B_{m})\,.
\end{split}
\label{eqA2l}
\end{equation}
\begin{prop}
\label{proGlue}
One has the commutative triangle
\begin{equation}
\xymatrix{
& P^{n}(c)_{ml} \ar[ld]_{\zeta_{m,l}} \ar[rd]^{\zeta_{l,m}}\\
\Hc_{m,l}\times\GL(c, \Com) \ar[rr]^-{\omega_{lm}} & & \Hc_{l,m}\times\GL(c, \Com)\,,
}
\label{eqtri}
\end{equation}
where $\zeta_{m,l}$ and $\zeta_{l,m}$ are the restrictions of $\zeta_{m}$ and $\zeta_{l}$, respectively, and
\begin{equation*}
\omega_{lm}(B_{m},E_{m},e;A_{2m})=\left(\tilde{\varphi}_{lm}(B_{m},E_{m},e),A_{2m}(c_{m-l}\bm{1}_{c}-s_{m-l}B_{m})\right)\,,
\end{equation*}
the functions $\tilde{\varphi}_{lm}$ being defined as in Proposition \ref{pro1}. The transition functions $\omega_{lm}$ are $\GL(c, \Com)^{\times 2}$-equivariant.
\end{prop}
\begin{proof}
We want to express $(B_l,E_l,e;A_{2l})$ in terms of $(B_m,E_m,e;A_{2m})$. We already have eq.~\eqref{eqA2l}; analogously, one can prove $A_{1l}=A_{2m}(s_{m-l}\bm{1}_{c}+c_{m-l}B_{m})$. From that, it follows that $B_{l}=(c_{m-l}\bm{1}_{c}-s_{m-l}B_{m})^{-1}(s_{m-l}\bm{1}_{c}+c_{m-l}B_{m})$.

As for $E_{l}$, one has
\begin{equation*}
\begin{split}
E_{l}&=\left[\sum_{p=1}^{n}\sigma^{n-1}_{-l;0,p-1} C_{p}\right]A_{2l}=\\
&=\left[\sum_{p=0}^{n-1} \sigma^{n-1}_{m-l;0p}B_{m}^{p}\right]E_{m}A_{2m}^{-1}A_{2l}=\\
&=(c_{l-m}\bm{1}_{c}-s_{l-m}B_{m})^{n}E_{m}\,,
\end{split}
\end{equation*}
where we have used eq. \eqref{eq64}, the relation $\sigma^{n-1}_{m-l}=\sigma^{n-1}_{-l}\sigma^{n-1}_{m}$ and Lemma \ref{lmP1}.

The equivariance of $\omega_{lm}$ is straightforward, and this completes the proof.
\end{proof}

By Proposition \ref{lmTomega} and Lemma \ref{lemmaquot} the immersion $\Hc\lhra \Hc\times\{\bm{1}_{c}\}$ induces an isomorphism
\begin{equation}
P^{n}(c)_{m}/\GL(c, \Com)^{\times 2}\simeq \Hc/\Delta\,,
\label{eqDelta}
\end{equation}
where $\Delta\subset\GL(c, \Com)^{\times 2}$ is the diagonal subgroup. By comparing eqs.~\eqref{eqGL(c)onK} and \eqref{eqrhol}, it turns out that $\Hc/\Delta=\Hc/\GL(c, \Com)$. It follows that
\begin{equation*}
P^{n}(c)_{m}/\GL(c, \Com)^{\times 2}\simeq \M^{n}(c)_{m}\,.
\end{equation*}
Recall   that  $\Hc$ is a  principal $\GL(c, \Com)$-bundle over $\Hc/\GL(c, \Com)$ \cite{Nakabook}. Now, by Proposition  \ref{lmTomega} there is an isomorphism $P^{n}(c)_{m}\simeq \Hc \times \GL(c, \Com)$ which is well-behaved with respect to the group actions;  as a consequence, $P^{n}(c)_{m}$ is a locally trivial principal $\GL(c, \Com)^{\times 2}$-bundle over $\M^{n}(c)_{m}$. 
Finally, Propositions \ref{pro1} and \ref{proGlue} ensure that this property globalizes, in the sense that $P^n(c)$ is a locally trivial principal $\GL(c, \Com)^{\times 2}$-bundle and this completes the proof of Theorem \ref{thm0}.

\bigskip
\section{Some geometrical constructions}
The projection $q_n\colon \tot(\Ol_{\Pu}(-n)) \lra \Pu$ induces a morphism 
$$p_{n,c}\colon \Hilb^c(\tot(\Ol_{\Pu}(-n))) \lra \mathbb P^c$$ 
defined as the composition
$$ \Hilb^c(\tot(\Ol_{\Pu}(-n))) \xrightarrow{\pi_{n,c}} S^c \tot(\Ol_{\Pu}(-n)) \xrightarrow{q_n^{(c)}} S^c \Pu = \mathbb P^c\,,$$
where $\pi_{n,c}$ is the Hilbert-Chow morphism. This morphism can be described in terms of ADHM data, as the following result essentially shows.
Let $N(c)$ be the space of pairs $(A_1,A_2)$ of $c\times c$ complex matrices satisfying property (P2), see the beginning of Section 
\ref{sectionmain}. The group $\GL(c, \Com)^{\times 2}$ acts on $N(c)$ as in equation \eqref{eqrho}.
\begin{prop} There is a commutative diagram of morphisms of schemes
\begin{equation}\label{diag}\xymatrix{
P^n(c) \ar[r] \ar[d]_{h_{n,c}} & \Hilb^c(\tot(\Ol_{\Pu}(-n)))  \ar[d]^{p_{n,c}} \\
N(c) \ar[r]^{g_c} & \mathbb P^c\,,} 
\end{equation}
where $g_c$ is the categorical quotient (in the sense of geometric invariant theory), and $h_{n,c}$, with reference to the notation in  the beginning of Section \ref{sectionmain}, is the morphism $$h_{n,c}(A_1,A_2;C_1,\dots,C_n;e) = (A_1,A_2)\,.$$
\end{prop}
\begin{proof} We introduce the open affine cover $\{U_m\}_{m=0,\dots c}$ of $\mathbb P^c$
$$ U_m = \{[x_0,\dots,x_c]\in\mathbb P^c \, \vert \, \sum_{p=0}^c \sigma^c_{m;p0}\,x_p \neq 0\}\simeq \mathbb C^c\,,$$
where the matrices $\sigma^c_{m}$ are defined in \eqref{eqsigma}. The inverse images $N_m=g_c^{-1}(U_m)$ yield  an affine open cover of $N(c)$.
The open  subsets $h_{n,c}^{-1}(N_m)\subset P^n(c)$ are exactly the sets $P^n(c)_m$ defined in equation \eqref{pcnm}. The composition 
$g_c\circ h_{n,c}$ on $P^n(c)_m$ can be identified with the map that to the quadruple $(B_m,E_m,e;A_{2m})$ (cf.\ Proposition \ref{lmTomega}) associates the evaluation of the symmetric elementary functions on the eigenvalues of $B_m$. Since checking the commutativity of the diagram \eqref{diag} is a local matter, and locally our ADHM data coincide with those for the Hilbert scheme of $\mathbb C^2$ as in    \cite{Nakabook},
we can proceed as in  \cite[p.\ 10]{Nakabook}.
\end{proof}

\bigskip
\appendix \section{Proofs of Propositions \ref{proopen} and \ref{pro1}}
In this Appendix, after giving some preliminary results, we provide proofs of  Propositions \ref{proopen} and \ref{pro1}.
 \subsection{A lemma about quotients}
 
If  $X$ is a smooth algebraic variety over $\Com$ with a (left) action
$\gamma\colon X\times G \to  X\times X$ 
 of a complex algebraic group $G$, the set-theoretical quotient $X/G$ has a natural structure of ringed space induced by the quotient map $q\colon X\lra X/G$.
 If the action is free, and the image of the morphism $\gamma$ 
is closed,   $X/G$ is a smooth algebraic variety, and the pair $(X/G,p)$ is a geometric quotient of $X$ modulo $G$. More precisely,
$X$ is a (locally isotrivial) principle $G$-bundle over $X/G$.
A proof of this fact was given in  \cite[Thm.~5.1]{bbr}.

Let $X$ be a smooth algebraic variety over $\mathbb C$, let $Y$ be a smooth closed subvariety; moreover let $G$ be a complex linear algebraic group, and $H$ a closed subgroup. Assume that $G$ acts on $X$ and $H$ on $Y$ so that the inclusion $j\colon Y \hookrightarrow X$ is $H$-equivariant.   We consider the quotients $q\colon X\lra X/G$ and $p\colon Y\lra Y/H$ as ringed spaces with the quotient topology, and structure sheaves given by the sheaves of invariant functions.

 \begin{lemma} \label{lemmaquot}
 Assume that the action of $G$ on $X$ is free, and that  the image of $\gamma\colon X\times G  \to X\times X$ is closed.
Moreover, assume that 
\begin{itemize} \item
the intersection of $\im j$ with every $G$-orbit in $X$ is nonempty;
\item
for all $G$-orbits $O_{G}$ in $X$, one has $\Stab_{G}(O_{G}\cap\im j)=\im\iota$.
\end{itemize}
Then $j$ induces an isomorphism $\bar{\jmath}$ of algebraic varieties between $Y/H$ and $X/G$.
\end{lemma}
\begin{cor} $X \to X/G$ and $Y\to Y/H$ are both principal bundles, and the second is a reduction of the structure group of the first.
If  $X \to X/G$  is locally trivial (and not only locally isotrivial), the same is true for $Y\to Y/H$.
\end{cor}
\begin{proof}
Since $\gamma$ is a closed immersion, it is proper. Hence by [11, Prop. 0.7] the morphism
$q$ is affine. Then if  $U\subset X/G$ is an open affine subset,   $V=q^{-1}(U)$ is   affine, $V=\Spec A$, so that $U= \Spec(A^{G})$, and the restricted morphism $q|_{V}$ is induced by the canonical injection
 $q^{\sharp}\colon A ^{G}\hookrightarrow A$. 
Since $j$ is an affine morphism \cite[Prop.~1.6.2.(i)]{EGA2}, the counterimage $W=j^{-1}(V)$ is affine, $W=\Spec B$, and by the equivariance of $j$ it is $H$-invariant. It follows that its image $p(W)=\Spec(B^{H})$ is affine, and the restricted morphism $p|_{W}$ is induced by the canonical injection $p^{\sharp}\colon B^{G}\hookrightarrow B$. 
Let $j^{\sharp}\colon A \to B $ be the homomorphism associated to $j$. It is easy to prove that $\im\left(j^{\sharp}\circ p^{\sharp}\right)\subseteq A^{G}$, and that this composition is an isomorphism, which induces $\bar{\jmath}$.
\end{proof}

\label{appBch}
\subsection{Premilinaries}
As we recalled in the Introduction, for any isomorphism class $[(\E, \theta)]$ in the moduli space $\Mk$ of framed sheaves on $\Sigma_n$, the underlying sheaf  $\E$ is isomorphic to the cohomology of a monad 
\begin{equation}\label{fundamentalmonad2}
\xymatrix{
M(\alpha,\beta):&0 \ar[r] & \Uk \ar[r]^-{\alpha} & \Vk \ar[r]^-{\beta} & \Wk \ar[r] & 0
}\,, 
\end{equation}
where $\vec{k} = (n,r,a,c)$.
To express the pair of morphisms $(\alpha,\beta)$ as a pair of matrices,
we select suitable bases for the space
\begin{align*}
\Hom&(\Uk,\Vk) \oplus\Hom(\Vk,\Wk)=\\
=&\left[\Hom\left(\Com^{k_1},\Com^{k_2}\right)\otimes H^{0}\left(\On(1,0)\right)\right]\oplus
\left[\Hom\left(\Com^{k_1},\Com^{k_4}\right)\otimes H^{0}\left(\On(0,1)\right)\right]\oplus\\
&\left[\Hom\left(\Com^{k_2},\Com^{k_3}\right)\otimes H^{0}\left(\On(0,1)\right)\right]\oplus
\left[\Hom\left(\Com^{k_4},\Com^{k_3}\right)\otimes H^{0}\left(\On(1,0)\right)\right],
\end{align*}
where the integers $k_1$, $k_2$, $k_3$, $k_4$ are specified in eq.~\eqref{k_i}. 
To this aim, after fixing homogeneous coordinates $[y_1,y_2]$ for $\Pu$, we introduce additional $c$ pairs of coordinates 
\begin{equation*}
[y_{1m},y_{2m}]=[c_{m}y_{1}+s_{m}y_{2},-s_{m}y_{1}+c_{m}y_{2}]\qquad m=0,\dots,c\,,
\label{beta10}
\end{equation*}
where  $c_m$ and $s_m$ are the real numbers defined in eq.~\eqref{eqcmsm}. The set $\left\{y_{2m}^{q}y_{1m}^{h-q}\right\}_{q=0}^{h}$ is a basis for $H^{0}\left(\On(0,h)\right)=H^{0}\left(\pi^{*}\Ol_{\Pu}(h)\right)$ for all $h\geq1$, where $\pi\colon\Sigma_{n}\longrightarrow\Pu$ is the canonical projection. Furthermore if we call $s_{E}$ the (unique up to homotheties) global section of $\On(E)$, it induces an injection $\On(0,n)\rightarrowtail\On(1,0)$, so that the set $\left\{(y_{2m}^{q}y_{1m}^{n-q})s_{E}\right\}_{q=0}^{n}\cup\{s_{\infty}\}$ is a basis for $H^{0}\left(\On(1,0)\right)$, where $s_{\infty}$ is the section characterized by the condition $\{s_{\infty}=0\}=\li$. We get
\begin{equation}
\begin{aligned}
\alpha&=
\begin{pmatrix}
\sum_{q=0}^{n}\alpha_{1q}^{(m)}\left(y_{2m}^{q}y_{1m}^{n-q}s_{E}\right)+\alpha_{1,n+1}s_{\infty}\\[7pt]
\alpha_{20}^{(m)}y_{1m}+\alpha_{21}^{(m)}y_{2m}
\end{pmatrix}\\
\beta&=
\begin{pmatrix}
\beta_{10}^{(m)}y_{1m}+\beta_{11}^{(m)}y_{2m} &
\sum_{q=0}^{n}\beta_{2q}^{(m)}\left(y_{2m}^{q}y_{1m}^{n-q}s_{E}\right)+\beta_{2,n+1}s_{\infty}
\end{pmatrix}\,.
\end{aligned}
\end{equation}
By restricting the display of the monad $M(\alpha,\beta)$ to $\li$,   twisting by $\Ol_{\li}(-1)$ and   taking cohomology, one finds the diagram
\begin{equation}
\xymatrix{
0 \ar[r] & H^{0}(\Vi(-1)) \ar[r] & H^{0}(\A_{\infty}(-1)) \ar[d]^-{\Phi} \ar[r] & H^{1}(\Ui(-1))  \ar[r] & 0\\
& & H^{0}(\Wi(-1))
}\,,
\label{eqdiaPhi}
\end{equation}
where $\A_{\infty}=(\coker\alpha)|_{\li}$. One of the conditions that characterize $\Lk$ is the invertibility of $\Phi$ (see \cite[\S 2, condition (c4)]{bbr}). By suitably splitting the short exact sequence which appears in \eqref{eqdiaPhi}, the morphism $\Phi$ becomes
\begin{equation*}
\Phi=
\begin{cases}
\beta_{11}^{(m)}\alpha_{10}^{(m)}+\beta_{21}^{(m)}\alpha_{20}^{(m)} &\text{for $n=1$;}\\[15pt]
\left(
\begin{array}{cccc|c}
 \beta_{10}^{(m)} &  &  &  &  \beta_{11}^{(m)}\alpha_{10}^{(m)}+\beta_{21}^{(m)}\alpha_{20}^{(m)}\\
 \beta_{11}^{(m)} & \beta_{10}^{(m)} & \text{\raisebox{0pt}[0pt][0pt]{\raisebox{.25cm}{\makebox[0pt]{\hspace{.7cm}\huge0}}}} &  &  \beta_{22}^{(m)}\alpha_{20}^{(m)}\\
 & \beta_{11}^{(m)} & \ddots &  &  \vdots \\
  & \text{\raisebox{0pt}[0pt][0pt]{\raisebox{-.4cm}{\makebox[0pt]{\hspace{-.9cm}\huge0}}}} & \ddots & \beta_{10}^{(m)} &  \beta_{2,n-1}^{(m)}\alpha_{20}^{(m)}\\
 & & & \beta_{11}^{(m)} &  \beta_{2n}^{(m)}\alpha_{20}^{(m)} 
\end{array}
\right)&\text{for $n>1$.}
\end{cases}
\end{equation*}

Let us now consider the principal $\GL(r,\Com)$-bundle $\tau\colon \Pk \lra \Lk$, whose fibre over a point $(\alpha,\beta)$ is naturally identified with the space of framings for the cohomology of the monad \eqref{fundamentalmonad2}.  By inspecting the display of $M(\alpha,\beta)$,  one sees  that  fixing a framing in the fibre $\tau^{-1}(\alpha,\beta)$ is equivalent to choosing a basis for $H^{0}\left(\ker\beta|_{\li}\right)=\ker H^{0}\left(\beta|_{\li}\right)$. So, $\Pk$ can be described as the quasi-affine variety of the triples $(\alpha,\beta,\xi)$, where $(\alpha,\beta)$ is a point of $\Lk$ and $\xi\colon\Com^{r}\lra V_{\vec{k}}:=H^{0}(\Vi)$ is an injective vector space morphism such that $H^{0}\left(\beta|_{\li}\right)\circ\xi=0$.
\subsection{Proof of Proposition \ref{proopen}}
\label{proof1}
We now are in the case where $r=1$ (hence, $a=0$).
We begin by constructing the immersion $j_{m}$ for any fixed $m\in\{0,\dots,c\}$. We define the morphism
\begin{equation*}
\begin{aligned}
\tilde{\jmath}_{m}\colon\End(\Com^{c})^{\oplus 2}\oplus\Hom(\Com^{c},\Com)&\lra \Hom(\Uk,\Vk)\oplus\Hom(\Vk,\Wk)\oplus\Hom(\Com^{r},V_{\vec{k}})\\
(b_{1},b_{2},e)&\longmapsto (\alpha,\beta,\xi)
\end{aligned}
\end{equation*}
\begin{equation*}
\text{where}\qquad
\left\{
\begin{aligned}
\alpha&=
\begin{pmatrix}
\bm{1}_c(y_{2m}^{n}s_{E})+\uside{t}\mspace{2mu}b_{2}s_\infty\\
\bm{1}_cy_{1m}+\uside{t}\mspace{2mu}b_{1}y_{2m}\\
0
\end{pmatrix}\\
\beta&=
\begin{pmatrix}
\bm{1}_cy_{1m}+\uside{t}\mspace{2mu}b_1y_{2m}&
-\left(\bm{1}_c(y_{2m}^{n}s_{E})+\uside{t}\mspace{2mu}b_2s_\infty\right)&
\uside{t}es_\infty
\end{pmatrix}\\
\xi&=
\begin{pmatrix}
0\\ \vdots \\ 0 \\
1
\end{pmatrix}
\end{aligned}
\right.
\label{eqP7}
\end{equation*}
and we let $j_{m}$ be the restriction of $\tilde{\jmath}_{m}$ to $\Hc$.
\begin{lemma}\label{lmclimm}
The morphism $j_{m}$ is a $\GL(c, \Com)$-equivariant closed immersion of $\Hc$ into $P_{\vec{k},m}$.
\end{lemma}
\begin{proof}
Since it is clear that $\tilde{\jmath}_{m}$ is a closed immersion, it is enough to prove that
\begin{equation*}
\im\tilde{\jmath}_{m}\cap P_{\vec{k},m}=\im j_{m}\,.
\end{equation*}

Let $(\alpha,\beta,\xi)=\tilde{\jmath}_{m}(b_{1},b_{2},e)$ be any point in the intersection $\im\tilde{\jmath}_{m}\cap P_{\vec{k},m}$;
the equation $\beta\circ\alpha=0$ implies that the triple $(b_{1},b_{2},e)$ satisfies the condition (T1), while the fact that $\beta\otimes k(x)$ has maximal rank for all $x\in\Sigma_{n}$ implies   condition (T2). It follows that
\begin{equation*}
\im\tilde{\jmath}_{m}\cap P_{\vec{k},m}\subseteq\im j_{m}\,.
\end{equation*}
To get the opposite inclusion, note that for all $(\alpha,\beta,\xi)\in\im\tilde{\jmath}_{m}$ the following conditions are satisfied:
\begin{itemize}
\item[(\textit{i})] the morphism $\alpha\otimes k(x)$ fails to have maximal rank at most at a finite number of points $x\in\Sigma_{n}$; hence, $\alpha$ is injective;
\item[(\textit{ii})] the morphisms $\alpha\otimes k(x)$ and $\beta\otimes k(x)$ have maximal rank for all points $x\in\li\cup f_{m}$;
\item[(\textit{iii})] the morphism $\Phi$ is invertible;
\item[(\textit{iv})] one has $\beta_{1}|_{f_{m}}=\bm{1}_{c}$;
\item[(\textit{v})] the morphism $\xi$ has maximal rank.
\end{itemize}
If $(\alpha,\beta,\xi)\in\im j_{m}$, the condition (T2) implies that $\beta\otimes k(x)$ has maximal rank for all $x\in\Sigma_{n}\setminus(\li\cup f_{m})$: by the condition (\textit{ii}) this is sufficient to ensure that $\beta$ is surjective. Condition (T1) implies that $\beta\circ\alpha=0$, so that we can define the quotient sheaf $\E=\ker\beta/\im\alpha$. By condition (\textit{i}) $\E$ is torsion free, by conditions (\textit{ii}) and (\textit{iii}) it is trivial at infinity, and by condition (\textit{iv}) $\E|_{f_{m}}$ is trivial as well.
The $\GL(c, \Com)$-equivariance of $j_{m}$ is readily checked.
\end{proof}
Lemma \ref{lemmaquot} will now allow us to prove that  $j_{m}$ induces an isomorphism between the quotients of $\Hc$ and $P_{\vec{k},m}$ under the actions of $\GL(c, \Com)$ and $\GL(c, \Com)^{\times 2}$, respectively. Thus, we have to show that for any $\Gk$-orbit $O_{\Gk}$ in $P_{\vec{k},m}$ the intersection $O_{\Gk}\cap\im j_{m}$ is not empty and that its stabilizer  in $\Gk$ coincides with $\im\iota$. To this aim, we build up a strictly descending chain of closed subvarieties
\begin{equation*}
\Pkm=:P^0\supsetneqq P^1 \supsetneqq \cdots \supsetneqq P^h=\im j_{m}\,,
\end{equation*}
for a certain $h>0$, such that there exists a strictly descending chain of subgroups
\begin{equation*}
\Gk=:G^0\supsetneqq G^1 \supsetneqq \cdots \supsetneqq G^h=\im\iota
\end{equation*}
with the property that $G^{i}$ is the stabilizer inside $\Gk$ of the intersection $O_{\Gk}\cap P^{i}$ for all $\Gk$-orbits in $P_{\vec{k},m}$.

Note that for each point $(\alpha,\beta,\xi)\in\Pk$ one has an exact sequence
\begin{equation*}
\xymatrix{
0 \ar[r] & \E \ar[r] & \On \ar[r] & \Ol_{Z} \ar[r] & 0
}
\end{equation*}
where $\E=\E_{\alpha,\beta}$ and $Z$ is the singular locus of $\E$. If we restrict this sequence to $f_{m}$,   twist it by $\Ol_{f_{m}}(-1)$ and take cohomology, we find out that $Z\cap f_{m}=\emptyset$ if and only if $H^{i}(\E|_{f_{m}}(-1))=0$ for $i=0,1$. By using the display of the monad $M(\alpha,\beta)$ one sees that this condition is equivalent to the condition $\det(\beta_{10}^{(m)})\neq0$ (the coefficient $\beta_{10}^{(m)}$ is defined in eq.~\eqref{beta10}).

By acting with $\Gk$ on $(\alpha,\beta,\xi)$ we can assume that
\begin{equation*}
\left\{
\begin{aligned}
\beta_{10}^{(m)}&=\bm{1}_{c}\\
\beta_{2q}^{(m)}&=0\qquad q=0,\dots,n-1\,.
\end{aligned}
\right.
\end{equation*}
These equations define the subvariety $P^{1}$, whose stabilizer $G^{1}$ is the subgroup of $\Gk$ determined by
the conditions 
$\psi_{11} =\chi$ and
$\psi_{12} =0$. 
Let $\uside{t}\mspace{2mu}b_1:=\beta_{11}^{(m)}$.
The equation $\beta\circ\alpha=0$ implies that
\begin{align}
\alpha_{1q}^{(m)}&=0\qquad q=0,\dots,n-1
\notag
\\
\alpha_{1n}^{(m)}&=-\beta_{2n}^{(m)}\alpha_{20}^{(m)}\,.
\label{eq3}
\end{align}
The invertibility of $\Phi$ is equivalent to the condition $\det\alpha_{1n}^{(m)}\neq0$, and by acting with $G^{1}$ we can assume that $\alpha_{1n}^{(m)}=\bm{1}_c$. This equation cuts the subvariety $P^{2}$ inside $P^{1}$, and the stabilizer $G^{2}$  is  the subgroup of $G^{1}$ where $\chi=\phi$.

From eq.~\eqref{eq3} we deduce that 
\begin{align}
\bm{1}_c&=-\beta_{2n}^{(m)}\alpha_{20}^{(m)}\,,\qquad\text{so that}\qquad \rk\beta_{2n}^{(m)}=\rk\alpha_{20}^{(m)}=c\,.
\label{eq4}
\end{align}
Therefore, by acting with $G^{2}$ we can assume that
\begin{equation*}
\alpha_{20}^{(m)}=
\begin{pmatrix}
\bm{1}_c\\
0
\end{pmatrix}\,.
\end{equation*}
This equation cuts the subvariety $P^{3}$ inside $P^{2}$, and the stabilizer $G^{3}$ is the subgroup of $G^{2}$,  where
\begin{equation*}
\psi_{22}=
\begin{pmatrix}
\phi & g_{12}\\
0 & g_{22}
\end{pmatrix}
\end{equation*}
for some $g_{12}\in\Hom(\Com,\Com^{c})$ and $g_{22}\in\Com^{*}$.

Eq.~\eqref{eq4} implies that $\beta_{2n}^{(m)}$ is of the form $\begin{pmatrix}
-\bm{1}_c & *
\end{pmatrix}$,
but by acting with   $G^{3}$ we can assume that $\beta_{2n}^{(m)}=
\begin{pmatrix}
-\bm{1}_c & 0
\end{pmatrix}$. This equation characterizes $P^{4}$ inside $P^{3}$, and the stabilizer $G^{4}$ is the subgroup of $G^{3}$ where $g_{12}=0$.
The equation $H^{0}\left(\beta|_{\li}\right)\circ\xi=0$ implies that
\begin{equation*}
\xi^{(m)}=
\begin{pmatrix}
0\\
\theta^{-1}
\end{pmatrix}\,.
\end{equation*}
By acting with $G^{4}$ we can assume that $\theta=1$: this cuts $P^{5}$ inside the variety $P^{4}$, and the stabilizer $G^{5}$ is the subgroups of $G^{4}$ where $g_{22}=1$. It is not difficult to show that $G^{5}$ concides with $\im\iota$.
To prove that $P^{5} = \im j_{m}$ we use once more  
the constraint $\beta\circ\alpha=0$ and get the system
\begin{equation*}
\left\{
\begin{aligned}
\uside{t}\mspace{2mu}b_1+
\begin{pmatrix}
-\bm{1}_c & 0
\end{pmatrix}
\alpha_{21}^{(m)}&=0\\
\alpha_{1,n+1}+\beta_{2,n+1}
\begin{pmatrix}
\bm{1}_c\\
0
\end{pmatrix}
&=0\\
\beta_{11}^{(m)}\alpha_{1,n+1}+\beta_{2,n+1}\alpha_{21}^{(m)}&=0\,.
\end{aligned}
\right.
\end{equation*}
From the first two equations we deduce that
\begin{equation*}
\alpha_{21}^{(m)}=
\begin{pmatrix}
\uside{t}\mspace{2mu}b_1\\
\uside{t}e_2
\end{pmatrix}
\qquad\text{and}\qquad\beta_{2,n+1}=
\begin{pmatrix}
-\alpha_{1,n+1} & \uside{t}e
\end{pmatrix}
\end{equation*}
for some $e\in\Hom(\Com^{c},\Com)$ and $e_{2}\in\Hom(\Com,\Com^{c})$. Only the last equation is not identically satisfied, and is equivalent to
\begin{equation}
\uside{t}\mspace{2mu}b_1\uside{t}\mspace{2mu}b_{2}-\uside{t}\mspace{2mu}b_{2}\uside{t}\mspace{2mu}b_1+\uside{t}e\uside{t}e_2=0\,,
\label{eqX'}
\end{equation}
where we have put $\uside{t}\mspace{2mu}b_{2}=\alpha_{1,n+1}$. Since the morphism $\beta\otimes k(x)$ has maximal rank for all $x\in\Sigma_{n}$, the quadruple $\left({}^{t}b_{1},{}^{t}b_{2},{}^{t}e,{}^{t}e_{2}\right)$ satisfies the hypotheses of \cite[Proposition 2.8]{Nakabook}, which implies $e_{2}=0$. It follows that $P^{5}=\im j_{m}$.

 \subsection{Proof of Proposition \ref{pro1}}
\label{proof2}
\begin{lemma}
For any $l,m=0,\dots,c$ and for any point $\vec{b}_{m}=(b_{1m},b_{2m},e_{m})\in \Hc_{m}$, there exists a unique element $\psi_{l}(\vec{b}_{m})=(\phi,\psi,\chi)\in\Gk$ such that
\begin{itemize}
\item
 $\chi=\bm{1}_{c}$;
\item
the point $(\alpha',\beta',\xi')=\psi_{l}(\vec{b}_{m})\cdot j_{m}(\vec{b}_{m})$ lies in the image of $j_{l}$.
\end{itemize}
If we set $(b_{1l},b_{2l},e_{l})=j_{l}^{-1}(\alpha',\beta',\xi')$, we have
\begin{equation}
\left\{
\begin{aligned}
b_{1l}&=\left(c_{m-l}\bm{1}_c-s_{m-l}b_{1m}\right)^{-1}\left(s_{m-l}\bm{1}_c+c_{m-l}b_{1m}\right)\\
b_{2l}&=\left(c_{m-l}\bm{1}_c-s_{m-l}b_{1m}\right)^{n}b_{2m}\\
e_{l}&=e_{m}\,.
\end{aligned}
\right.
\label{eqbbel}
\end{equation}
\end{lemma}
\begin{proof}
If we set $(\alpha,\beta,\xi)=j_{m}(\vec{b}_{m})$, by expressing  $[y_{1m},y_{2m}]$ as functions of $[y_{1l},y_{2l}]$ we get
\begin{equation*}
\begin{aligned}
\alpha&=
\begin{pmatrix}
\sum_{q=0}^{n}(\sigma_{q}\bm{1}_{c})(y_{2l}^{q}y_{1l}^{n-q}s_{E})+\uside{t}\mspace{2mu}b_{2m}s_\infty\\
d_{1m}y_{1m}+d_{2m}y_{2m}\\
0
\end{pmatrix}\,,\\
\beta&=
\begin{pmatrix}
d_{1m}y_{1m}+d_{2m}y_{2m}&
-\sum_{q=0}^{n}(\sigma_{q}\bm{1}_{c})(y_{2l}^{q}y_{1l}^{n-q}s_{E})-\uside{t}\mspace{2mu}b_{2m}s_\infty&
\uside{t}e_{m}s_\infty
\end{pmatrix}\,,
\end{aligned}
\end{equation*}
where
\begin{equation*}
d_{1m}=c_{m-l}\bm{1}_c-s_{m-l}\uside{t}\mspace{2mu}b_{1m}\qquad\qquad d_{2m}=s_{m-l}\bm{1}+c_{m-l}\uside{t}\mspace{2mu}b_{1m}
\end{equation*}
and we have put $\sigma_{q}=\sigma^{n}_{l-m;nq}$ for $q=0,\dots,n$ (see eq. \eqref{eqsigma}). The explicit form of $\psi_{l}(\vec{b}_{m})$ is obtained by imposing the equality 
\begin{equation}
(\phi,\psi,\bm{1}_{c})\cdot(\alpha,\beta,\xi)=j_{l}(b_{1l},b_{2l},e_{l})
\label{eqbpsi}
\end{equation}
for some $(b_{1l},b_{2l},e_{l})\in \Hc_{l}$. One gets
\begin{gather*}
\begin{aligned}
\phi&=d_{1m}^{-(n-1)}\\
\psi&=
\begin{pmatrix}
d_{1m} & \psi_{12,1} & 0\\
0 & d_{1m}^{-n} & 0\\
0 & 0 & \bm{1}_{r}
\end{pmatrix}\,,
\end{aligned}
\\[5pt]
\text{where}\qquad\psi_{12,1}=-\sum_{q=0}^{n-1}\sum_{p=0}^{q}\sigma_{q-p}\left(-d_{2m}d_{1m}^{-1}\right)^{p}y_{1l}^{q}y_{2l}^{n-1-q}\,.
\end{gather*}
Eq.~\eqref{eqbbel} follows    from eq.~\eqref{eqbpsi}.
\end{proof}
Since $j_{m}$ and $j_{l}$ are injective, the map $\vec{b}_{m}\longmapsto\psi_{l}(\vec{b}_{m})\cdot\vec{b}_{m}$ induces the morphism $\tilde{\varphi}_{lm}$ in eq.~\eqref{eqfiml}. This completes the proof of Proposition \ref{pro1}.

\frenchspacing\bigskip
 %
%
%\bibliographystyle{siam}
%\bibliography{bibliografia}
%
 
 \def\cprime{$'$} \def\cprime{$'$} \def\cprime{$'$} \def\cprime{$'$}

\end{document}